\documentclass[11pt,twoside,a4paper]{article}
\usepackage{amsmath,amssymb,amsthm}
\usepackage{hyperref}
\usepackage{graphicx,psfrag}
\usepackage{tikz}

\newtheorem{thm}{Theorem}[section]
\newtheorem{lem}[thm]{Lemma}

\newtheorem{cor}[thm]{Corollary}

\theoremstyle{definition}

\newtheorem{defi}[thm]{Definition}

\theoremstyle{remark}

\newcommand{\R}{\mathbb{R}}

\newcommand{\N}{\mathbb{N}}

\newcommand{\cM}{\mathcal{M}}

\newcommand{\ga}{\gamma}

\newcommand{\om}{\omega}
\newcommand{\Om}{\Omega}

\newcommand{\la}{\lambda}

\renewcommand{\phi}{\varphi}

\newcommand{\crt}{\operatorname{crt}}

\newcommand{\diam}{\operatorname{diam}}

\newcommand{\CAT}{\operatorname{CAT}}

\newcommand{\id}{\operatorname{id}}

\newcommand{\Trip}{\operatorname{Trip}}

\renewcommand{\d}{\partial}
\newcommand{\di}{\d_{\infty}}

\newcommand{\sm}{\setminus}
\newcommand{\sub}{\subset}

\newcommand{\ov}{\overline}

\newcommand{\pari}{\partial_{\infty}}

\begin{document}

\title{Trees and ultrametric M\"obius structures}
\author{Jonas Beyrer
\ \& Viktor Schroeder\footnote{Supported by Swiss National
Science Foundation Grant 153607}}

\date{}
\maketitle

\begin{abstract}
We define the concept of an ultrametric M\"obius space and use this to 
characterize nonelementary geodesically complete trees.
\end{abstract}


\medskip


\section{Introduction}
\label{sect:introduction}

In this paper we define the concept of an ultrametric M\"obius space $(Z,\cM)$, where $Z$ is some set with cardinality
$|Z|\geq 3$ and $\cM$ an ultrametric M\"obius structure (see section \ref{sec:M-structure}). 
These M\"obius spaces describe in a natural way the asymptotic geometry of trees.
To a metric tree $(X,d)$ one can associate in a natural way a boundary at infinity $\di X$. 
We call a tree $(X,d)$ geodesically complete, if every geodesic segment can be extended to a complete geodesic line.
We call a tree nonelementary, if $|\di X|\geq 3$.
The boundary $\di X$ of a nonelementary tree $(X,d)$
carries a canonical ultrametric
M\"obius structure $\cM_X$.
If $F:(X,d)\to (X',d')$ is an isometric embedding of trees, then $F$ extends in a natural way to a M\"obius map
$f:(\di X,\cM_X) \to (\di X,\cM_{X'})$.

We show that the association 
$(X,d) \mapsto (\di X, \cM_X)$ defines an equivalence between
the category of isometry classes of nonelementary geodesically complete trees
and the category of M\"obius classes of ultrametric M\"obius spaces.
Therefore we will prove the following facts:

\begin{enumerate}
 \item Let $(Z,\cM)$ be a complete ultrametric M\"obius space. 
Then there exists a unique nonelementary geodesically complete tree $(X,d)$, such that $(\partial_{\infty} X,\cM_{X})$
is M\"obius equivalent to $(Z,\cM)$.

\item Let $(X,d)$ and  $(X',d')$ be nonelementary geodesically complete trees. Then for
any M\"obius embedding $f:(\pari X,\cM_X) \to (\pari X',\cM_{X'})$, there exists a unique isometric embedding $F:(X,d)\to (X',d')$
such that the extension of $F$ to the boundary coincides with $f$.
If $f$ is in addition surjective, then $F$ is an isometry.   
\end{enumerate}

These results are in the spirit of \cite{Hu} ,who investigated the case of trees with a basepoint.
Probably all of the results of this paper are known in some way but usually formulated in a different language. 
To our opinion the viewpoint
using ultrametric M\"obius structures is new and worth to be formulated.
Our paper is also inspired by the work of \cite{Bi}. In this paper the author associates to a certain M\"obius
space $(Z,\cM)$ a filling $(X,d)$ in the more general context of $\CAT(-1)$ spaces. The idea in \cite{Bi} is to
define the filling as the subset $\cM^a_1 \subset \cM_X$ of
antipodal diameter 1 metrics in $\cM_X$ with a canonical metric 
$d_{\cM^a_1}$ on it (compare section \ref{sec:ad1}). Indeed one obtains the following fact (which is a slight generalization of
a result in \cite{Bi}):

\begin{itemize}
 \item[3.] Let $(Z,\cM)$ be a complete ultrametric M\"obius space,  $\cM^a_1 \subset \cM_X$ the subset of
antipodal diameter 1 metrics and $(X,d)$ the tree from fact 1. Then $i:(X,d)\to (\cM^a_1,d_{\cM^a_1})$, $x\mapsto \rho_x$ is an isometry.
Here $\rho_x$ is the Bourdon metric at $x$. 
\end{itemize}

The structure of the paper is as follows. In section \ref{sec:M-structure} we introduce the concept of
ultrametric M\"obius spaces, in section \ref{sec:t} we recall some facts about metric trees. 
In section \ref{sec:tree-versus-M} we prove facts 1 and 2 and in the final section \ref{sec:ad1} we show
fact 3.

\section{Ultrametric M\"obius spaces}\label{sec:M-structure}

\subsection{M\"obius Structure} 

Let
$Z$ 
be a set with cardinality 
$|Z| \ge 3$.
An {\em extended metric} on 
$Z$ 
is a map 
$\rho:Z\times Z \to [0,\infty ]$,
such that
there exists
a set
$\Om(\rho) \sub Z$ with
cardinality
$ |\Om(d)| \in \{ 0,1\}$, 
such that 
$\rho$
restricted to the set
${Z\sm \Om(\rho)}$ 
is a metric 
(taking only values in $[0,\infty)$) and such that
$\rho(z,\omega)=\infty$ 
for all 
$z\in Z\setminus \Om(\rho)$, $\om \in \Om(d)$. 
Furthermore $d(\omega,\omega)=0$.\\
If
$\Om(\rho)$
is not empty,
we sometimes denote
$\om \in \Om(\rho)$ simply as
$\infty$ and call it the
(infinitely) {\em remote point}
of
$(Z,\rho)$.
We often write also
$\{\om\}$ for the set
$\Om(\rho)$ and
$Z_{\om}$ for the set
$Z\sm \{\om\}$.

We consider on $(Z,\rho)$  the topology with the 
basis consisting of all open distance balls $B_{r}(z)$
around points in $z\in Z_{\om}$ and the
complements $D^C$ of all closed distance balls $D=\overline{B}_r(z)$.
Note that $(Z,\rho)$ is a Hausdorff space.\\

We call an extended metric space {\em complete}, if first every
Cauchy-sequence in $Z_{\om}$ converges and secondly
if the infinitely remote point $\om$ exists in case that $Z_{\om}$
is unbounded.
For example the real line
$(\R,\rho)$,
with its standard metric is {\em not} complete as an extended metric space,
while
$(\R\cup\{\infty\},\rho)$ is complete.

One can also characterize completeness using the following notation.

\begin{defi}
Let $(Z,\rho)$ be an extended metric space. A sequence $(z_i)\subset Z$ is a \emph{Cauchy-sequence in $(Z,\rho)$}, if one the following is true:
\begin{enumerate}
\item For an arbitrary point $o\in Z_{\om}$ there exists $C>0$ and $n\in \N$, such that $\rho(o,z_i) \leq C $ for 
all $i\geq n$ and the sequence $(z_i)_{i\geq n}$ is a Cauchy-sequence 
in $Z_{\om}$ in the classical sense.
\item For an arbitrary point $o\in Z_{\om}$ and every $C\in \R$, there is an $n\in\N$, such that $\rho(o,z_i)\geq C$ for all $i\geq n$.
\end{enumerate}
\end{defi}

Then completeness of $(Z,\rho)$ is equivalent to the convergence of Cauchy-sequences. A Cauchy-sequence of the second type is converging to 
the remote point $\om$.

We say that a quadruple 
$(x,y,z,w)\in Z^4$ is 
{\em admissible}, if no entry occurs three or
four times in the quadruple.
We denote with
$Q\sub Z^4$
the set of 
admissible quadruples.
We define the {\em cross ratio triple} as the map
$\crt:\ Q \to \Sigma \subset \R P^2$ 
which maps 
admissible quadruples to points in the real projective plane defined by
$$\crt_{\rho}(x,y,z,w)=(\rho(x,y)\rho(z,w): \rho(x,z)\rho(y,w) : \rho(x,w)\rho(y,z)),$$
here 
$\Sigma$ 
is the subset of points 
$(a:b:c) \in \R P^2$, 
where
all entries 
$a,b,c$ 
are nonnegative or all entries are non-positive.
Note that 
$\Sigma$ can be identified with the standard $2$-simplex,
$\{(a,b,c)\, |\, a,b,c \ge 0,\, a+b+c=1\}$.

We use the standard conventions for the calculation with 
$\infty$.
If 
$\infty$ occurs once in 
$Q$, say 
$w=\infty$,
then
$\crt_{\rho}(x,y,z,\infty)=(\rho(x,y):\rho(x,z):\rho(y,z))$.
If 
$\infty$ 
occurs twice , say $z=w=\infty$ then
$\crt_{\rho}(x,y,\infty,\infty)=(0:1:1)$.

The cross ratio triple is a more symmetric way to define the classical
cross ratio $[.,.,.,.]_{\rho}:Q\to [0,\infty]$ which we define as
$$[x,y,z,w]_{\rho}:=\frac{\rho(x,z)\rho(y,w)}{\rho(x,y)\rho(z,w)}.$$

It is not difficult to check that 
$\crt_{\rho}:Q\to \Sigma$ 
is continuous,
where $Q$ and $\Sigma$ carry the obvious topologies induced by
$(Z,\rho)$ and $\R P^2$. Thus, if $(x_i,y_i,z_i,w_i) \in Q$ for $i \in \N$ and assume
$x_i\to x,\ldots,w_i\to w$, where $(x,y,z,w)\in Q$ then 
$\crt_{\rho}(x_i,y_i,z_i,w_i)\to\crt_{\rho}(x,y,z,w)$.

One can characterize convergence in $(Z,\rho)$ and the Cauchy-sequence property in terms of the cross ratio by
the following Lemma (where we omit the proof) :

\begin{lem} \label{lem:Cauchy-sequence}
Let $(Z,\rho)$ be an extended metric space. 
A sequence $(z_i) \subset Z$ converges to $z\in Z$, if and only if there are distinct points  $a,b \in Z$, s.t. $\crt_{\rho}(z,z_j,a,b)\to (0:1:1)$. \\
A sequence $(z_i)\subset Z$ is a Cauchy-sequence in $(Z,\rho)$, if and only if
there are distinct points $a,b\in Z$, s.t. $\crt_{\rho}(z_i,z_j,a,b)\to (0:1:1)$ for $i,j\to\infty$.
\end{lem}

A map
$f:(Z,\rho)\to (Z',\rho')$
between two extended metric spaces
is called
{\em M\"obius}, if 
$f$ is injective and for all
admissible quadruples
$(x,y,z,w)$ of
$Z$,
$$\crt_{\rho'}(f(x),f(y),f(z),f(w))=\crt_{\rho}(x,y,z,w).$$

Now Lemma \ref{lem:Cauchy-sequence} implies that M\"obius maps are continuos and Cauchy-sequences are mapped to Cauchy-sequences.

Two extended metric spaces
$(Z,\rho)$ and
$(Z',\rho')$ are
{\em M\"obius equivalent},
if there exists a bijective
M\"obius map
$f:Z\to Z'$.
In this case also
$f^{-1}$ is a M\"obius map and
$f$ is in particular a
homeomorphism.
Furthermore, by Lemma \ref{lem:Cauchy-sequence} $(Z,\rho)$ is complete if and only if $(Z',\rho')$ is.\\
We say that two extended metrics
$\rho$ and $\rho'$
on a set 
$Z$ are
{\em M\"obius equivalent},
if the identity map
$\id:(Z,\rho)\to (Z,\rho')$
is a M\"obius map.
M\"obius equivalent metrics define
the same topology on
$Z$ and $(Z,\rho)$ is complete if and only $(Z,\rho')$ is.

A {\em M\"obius structure} on a set
$Z$ is a nonempty set
$\cM$ of extended metrics on
$Z$,
which are
pairwise M\"obius equivalent and which is
maximal with respect to that property.
Thus $\cM$ is just an equivalence class of M\"obius equivalent
extended metrics on $Z$.

Given an extended metric $\rho$ on $Z$, we denote by
$[\rho]$ the equivalence class of all M\"obius equivalent extended metrics.
Then $[\rho]$ is a M\"obius structure on $Z$.

A {\em M\"obius space} is a pair $(Z,\cM)$ of a set $Z$ with cardinality $|Z|\ge 3$ and
a M\"obius structure $\cM$ on $Z$.
If $(Z,\cM)$ is a M\"obius space and $(x,y,z,w)$ an admissible quadruple, then
the cross ratio triple is independent of the metric and hence
$\crt(x,y,z,w) := \crt_{\cM}(x,y,z,w): = \crt_{\rho}(x,y,z,w)$ is well defined, where $\rho \in \cM$.

A M\"obius space $(Z,\cM)$ is a topological space, since $\rho , \rho' \in \cM$ define the same topology on $Z$.
A M\"obius space is {\em complete}, if $(Z,\rho)$ is complete for $\rho \in \cM$.

A map $f:(Z,\cM) \to (Z',\cM')$ between two M\"obius spaces is called {\em M\"obius}, if it is injective
and preserves the cross ratio triple. The spaces are {\em M\"obius equivalent} if there exists a
bijective M\"obius map between them.

Similar as for metric spaces, one can define the completion of
a M\"obius space $(Z,\cM)$. Define the set $\ov{Z}$ to be the set of equivalence classes
of Cauchy-sequences, where two Cauchy-sequences $(x_i)$ and $(y_i)$ are called equivalent,
if $\crt(x_i,y_i,a,b) \to (0:1:1)$ for some distinct $a,b$. If $\rho \in \cM$ then define an
extended metric $\ov{\rho}$ on $\ov{Z}$ by
$\ov{\rho}((x_i),(y_i)) := \lim \rho(x_i,y_i)$, here $(x_i)$ and $(y_i)$ are nonequivalent
Cauchy-sequences. The completion of $(Z,\cM)$ is unique up to M\"obius equivalence.

\subsection{Ultrametric M\"obius structures}

A point $(a:b:c)\in \Sigma \subset \R P^2$, $a,b,c \geq 0$ is called {\em ultrametric}, if
the two largest of the numbers $a,b,c$ coincide. 
The set of ultrametric points in $\Sigma$ is
{\em Y-shaped} and consists out of the three affine segments from the central point
$(1:1:1)$ to the three points $(0:1:1), (1:0:1), (1:1:0)$.

A M\"obius space $(Z,\cM)$ is called {\em ultrametric}, if for all admissible quadruples
$(x,y,z,w)$ the cross ratio triple
$\crt(x,y,z,w)$ is ultrametric.

As usual, one calls an extended metric $\rho$ on a set $Z$ {\em ultrametric}, if for 
distinct $x,y,z \in Z_{\om}$ the point
$(\rho(x,y):\rho(x,z):\rho(y,z))$ is ultrametric.

If $(Z,\rho)$ is an extended ultrametric space and $\cM = [\rho]$, then
$(Z,\cM)$ is an ultrametric M\"obius space. This follows e.g. from \cite{BS} Lemma 5.1.2.
On the other side, if $(Z,\cM)$ is ultrametric, then not necessarily all $\rho \in \cM$ are ultrametrics.
Consider for example the set $Z=\{x,y,z,w\}$ and let
$\rho(x,y) = t$, $\rho(z,w) =1/t$, $\rho(x,z)=s$, $\rho(y,w)=1/s$, $\rho(x,w)= r$, $\rho(y,z)=1/r$, where
$t,s,r$ are numbers close to one. The only relevant cross ratio triple is
$\crt_{\rho}(x,y,z,w)=(1:1:1)$ which is ultrametric. All these metrics are M\"obius equivalent but only
for $s=t=r=1$ this is an ultrametric.

However in two cases one obtains ultrametrics:

\begin{lem} \label{omega-metrics-are-ultrametric}
 Let $(Z,\cM)$ be an ultrametric M\"obius space and $\rho \in \cM$ such that a remote point
 $\om\in\Om(\rho)$ exists. Then $\rho$ is an ultrametric.
\end{lem}

\begin{proof}
Consider without loss of generality $x,y,z\in Z_{\om}$. Then
\begin{align*}
crt_{\rho}(x,y,z,\omega)=(\rho(x,y):\rho(x,z):\rho(y,z)).
\end{align*}
As $crt_{\rho}$ is ultrametric it follows that $\rho$ is an ultrametric.
\end{proof}

A metric $\rho$ on $Z$ is called {\em antipodal diameter 1 metric}, if
$\diam (Z,\rho) =1$ and every $z\in Z$ has an antipodal point $z'\in Z$, i.e. $\rho(z,z')= 1$.

\begin{lem}\label{diam1-antipodal-metrics-are-ultrametric}
Let $(Z,\cM)$ be an ultrametric M\"obius space and $\rho \in \cM$ an antipodal diameter 1 metric, 
then $\rho$ is an ultrametric.
\end{lem}

\begin{proof}
Assume the contrary, then there are
$x,y,z\in Z$ such that $\rho(x,y)>\rho(x,z)\geq \rho(y,z)$. Let $z^{'}\in Z$ with $\rho(z,z^{'})=1$. Then 
\begin{align*}
\rho(x,y)\rho(z,z^{'})>\rho(x,z)\rho(y,z^{'})\quad\wedge\quad \rho(x,y)\rho(z,z^{'})>\rho(y,z)\rho(x,z^{'}),
\end{align*}
which contradicts the fact that the cross ratio triple $\crt(x,y,z,z')$ is ultrametric. 
\end{proof}

An ultrametric M\"obius space is in ptolemaic in the sense of \cite{FS}, which implies:

\begin{lem}\label{ultrametric-is-ptolemaic}
Let $(Z,\cM)$ be an ultrametric M\"obius space and let $\om \in Z$, then there exists $\rho \in \cM$, such that
$\Om(\rho) = \{\om\}$.
\end{lem}

\section{Trees} \label{sec:t}

For this paper a {\em tree} is a metric space
$(X,d)$ with the following two properties:

\begin{enumerate}

 \item $X$ is {\em uniquely geodesic} : if $x,y \in X$, then there exists a unique isometric map
 $c:[0,\ell] \to X$, $\ell=d(x,y)$ with $c(0)=x$ and $c(\ell)=y$. The image of $c$ is denoted by
 $[x,y]$ and called the (geodesic) {\em segment} between $x$ and $y$. The points $x$ and $y$ are called the 
 {\em endpoints} of $[x,y]$. Note that $[x,y] =[y,x]$.
 
 \item If two geodesic segments $[x,y]$ and $[y,z]$ have only one point in common, which is an endpoint of each segment,
 then their union is also a segment; i.e.: if $[x,y]\cap [y,z] =\{y\}$ then 
 $[x,z]=[x,y] \cup [y,z]$.
\end{enumerate}

For the distance in $X$ we also use the notation $|xy|:= d(x,y)$. 

A tree is called {\em geodesically complete}, if every nontrivial segment
is contained in a geodesic line, i.e. if every isometric map 
$[0,\ell]\to X$, $\ell > 0$ extends to an isometric map 
$\R \to X$.

Trees are $\CAT(-1)$ spaces in the sense of comparison geometry (see e.g. \cite{BS}) and are in particular
geodesic Gromov hyperbolic spaces. Actually they are $0$-hyperbolic and $\CAT(\kappa)$ for all $\kappa \in \R$.
We now collect some well known facts about trees.

If $x,y,z \in X$ are three points, then there exists a unique point
$w := \Trip(x,y,z)\in X$ with 
$\{ w\} =[x,y]\cap [y,z] \cap [x,z]$.
The point $w$ is called the {\em tripod} of $x,y,z$.

We have 
$|xw|=(z|y)_x$, 
$|yw|=(x|z)_y$ and
$|zw|=(x|y)_z$, where we use the so called {\em Gromov product}
$(x|y)_z := \frac{1}{2}(|zx|+|zy|-|xy|)$.

To define the boundary at infinity we call 
a sequence $(x_i)$ in $X$ a {\em Gromov sequence}, if for some 
basepoint $o\in X$, $(x_i|x_j)_o \to \infty$ for $i,j \to \infty$.
Two Gromov sequences $(x_i) ,(y_i)$ are called equivalent, if 
$(x_i|y_i)_o\to \infty$. The equivalence classes are the points
of the boundary $\pari X$.

We call a tree $X$ {\em non-elementary}, if the cardinality  $|\pari X| \geq 3$.
Thus (up to isometry) the only nontrivial elementary geodesically complete tree is
the euclidean line $\R$.

Given a point $a \in \pari X$ and a point $x \in X$, there exists a unique
geodesic ray (i.e. a unique isometric map) $c:[0,\infty) \to X$, such that
$c(0)=x$ and $c(\infty)=a$ (which means that the sequence $c(i)$, $i\in \N$, represents $a$).
We denote the image of this ray as 
$[x,a)$.

Given two distinct points $a,b \in \pari X$ and $x \in X$, there exists a unique point
$u \in X$ such that $[x,a) \cap [x,b) = [x,u]$. Then $ (a,b) := [u,a) \cup [u,b)$ is the image of
an isometric embedding $\R \to X$ and called the line between $a$ and $b$.
If $(a_i), (b_i)$ are sequences in $X$ with $a_i \to a$ and $b_i \to b$, then
$(a_i|b_i)_x = |xu|$ for $i$ sufficiently large. We define $(a|b)_x := \lim (a_i|b_i)_x = |xu|$.
For $x,y \in X$ and $a \in \pari X$ we also define $(a|x)_y := \lim (a_i|x)_y$, where $(a_i)$ is
a sequence in $X$ converging to $a$.

If $a \in \pari X$, the Busemann function
$B_a:X\times X \to \R$ is defined as
$$B_a(x,y)= \lim ( |xa_i|-|ya_i| ) = (a|y)_x -(a|x)_y$$
where $(a_i)$ is a sequence converging to $a$. All these limits are well defined.

For $x \in X$ and $a,b \in \pari X$ let
$\rho_x(a,b) := e^{-(a|b)_x}$. The $\rho_x$ defines the {\em Bourdon metric }  (compare \cite{Bou} ) on $\pari X$ with basepoint 
$x \in X$. This is an ultrametric. 
Since
$(a|b)_x-(a|b)_y= \frac{1}{2}(B_a(x,y)+B_b(x,y))$,
we have 
$$\rho_y(a,b)= \la(a) \la(b) \rho_x(a,b)$$
where $\la=\la_{x,y}:\pari X\to \R$ is $\la(a)=e^{\frac{1}{2}B_a(x,y)}$.

In particular the metrics $\rho_x$ and $\rho_y$ are M\"obius equivalent and they define
(in the case that $X$ is nonelementary) the same M\"obius structure on $\pari X$.

We denote this M\"obius structure by $\cM_{X}$ and call it the {\em canonical M\"obius structure} on $X$.
This M\"obius structure is ultrametric.

For $\om \in \pari X$ and $o \in X$ we define for $a,b \in \pari X \setminus \{\om\}$ 
$$(a|b)_{\om,o} := (a|b)_o-(a|\om)_o-(b|\om)_o.$$
Then
$$\rho_{\om,o}(a,b) := e^{-(a|b)_{\om,o}} = \frac{\rho_o(a,b)}{\rho_o(a,\om) \rho_o(b,\om)}$$
defines a metric on $\pari X \setminus \{\om\}$ which can be considered as an extended metric on $\pari X$ with
remote point $\om$. This metric is also contained in $\cM_X$.

\section{Geodesically complete trees versus ultrametric M\"obius spaces} \label{sec:tree-versus-M}

\begin{thm}\label{filling-thm}
Let $(Z,\cM)$ be a complete ultrametric M\"obius space. 
Then there is a nonelementary geodesically complete tree X, such that $(\partial_{\infty} X,\cM_{X})$
is M\"obius equivalent to $(Z,\cM)$.
\end{thm}

\begin{proof}
By Lemma \ref{ultrametric-is-ptolemaic} we know that for an 
ultrametric M\"obius structure $(Z,\cM)$ there is $\rho\in\cM$ such that $\Omega(\rho)=\lbrace\omega\rbrace$.
By Lemma \ref{omega-metrics-are-ultrametric} $\rho$ is defines an ultrametric on $Z_{\om}$. We define  the space
$$X:=(Z_{\omega}\times\R)\slash \sim\quad$$
where
$$ (z_1,t_1)\sim (z_2,t_2) \Longleftrightarrow t_1=t_2\:\wedge \: t_1\leq -\ln \rho(z_1,z_2)$$
for $z_1,z_2\in Z_{\omega}$ and $t_1,t_2\in\R$. 
Since $\rho$ is an ultrametric $\sim$ is an equivalence relation.
The space $X$ will be the "filling" of $Z$.  Therefore we define 
\begin{align*}
d([z_1,t_1],[z_2,t_2]):=t_1+t_2-2\min(t_1,t_2,-\ln \rho(z_1,z_2)).
\end{align*}
Similar to section 6 of \cite{Hu} one can show that $d$ is a well defined metric.\\ 
We first show that $(X,d)$ 
has a natural bicombing, i.e. given
$[z_1,t_1], [z_2,t_2]\in X$, there is a {\em natural geodesic} between these points, 
which can be described as follows:\\
Assume w.l.o.g. $t_1<t_2$. If $t_1\leq -\ln \rho(z_1,z_2)$, the natural geodesic 
joining $[z_1,t_1]$ to $[z_2,t_2]$ is $\gamma: [t_1,t_2]\to X$, $\gamma(s):= [z_2,s]$. 
If $t_1> -\ln \rho(z_1,z_2)$, let $a:= t_1+\ln \rho(z_1,z_2) +t_2+\ln \rho(z_1,z_2)$. 
We define the natural geodesic joining $[z_1,t_1]$ to $[z_2,t_2]$ to be
\begin{align*}
\gamma:[0,a]\to X \quad \gamma(s):=\Bigg\lbrace
\begin{matrix}
[z_1,t_1-s]\hspace{1.8cm} & s\leq t_1+\ln \rho(z_1,z_2)\\
[z_2,-\ln \rho(z_1,z_2)+s] & s> t_1+\ln \rho(z_1,z_2)
\end{matrix}
\end{align*}
It is straight forward to check that $\gamma$ is a geodesic with  $\gamma(0)=[z_1,t_1]$ and $\gamma(a)=[z_2,t_2]$. 
The bicombing implies in particular that $X$ is geodesic. 
For any three points $[z_1,t_1]$, $[z_2,t_2]$, $[z_3,t_3]\in X$ there is a geodesic triangle formed by natural geodesics. 
Using the ultrametric property of $\rho$ we can assume w.l.o.g. $\rho(z_1,z_2)\leq \rho(z_1,z_3)=\rho(z_2,z_3)$. 
By definition of the natural geodesics it follows that this geodesic triangle is $0$-thin in the sense of Gromov and the 
tripod is $[z_1,-\ln \rho(z_1,z_2)]$. In particular for any three points $x,y,z \in X$, there exists a tripod $u$ which is
contained in all natural geodesics between the three points. Thus the equality $d(x,z) + d(z,y) = d(x,y)$ holds 
only if $z=u$ and hence $z$ lies on the natural
geodesic from $x$ to $y$. This implies that the natural geodesics are (up to parametrization) the unique geodesics. 
Thus $X$ is uniquely geodesics and satisfies the first defining property of a tree.
By definition of the natural geodesics it is easy to check that also the second defining property of a tree is satisfied.

In order to prove that $X$ is geodesically complete, we will show that any two points lie on a geodesic line. As both types of natural geodesics can be extended to bi-infinite lines,
namely $\gamma$ is either of the form $\gamma(t)=[z,t]$ for some $z\in Z_{\omega}$ or $\gamma(t)=[z_1,-t]$ for $t<-\ln\rho(z_1,z_2)$ and $\gamma(t)=[z_2,t]$ 
for $t\geq -\ln\rho(z_1,z_2)$, $X$ is geodesically complete.\\ 

We will show that the geodesic boundary of $X$ equals $Z$. 
We define a map 
$\iota:Z \to \di X$ in the following way:
For $z\in Z_{\om}$ let $\iota (z) :=\ga_z(\infty)$ where
$\gamma_z (t):=[z,t]$. 
Furthermore $\iota(\omega) := \ov{\ga}_z(\infty)$ where $\ov{\ga}_{z}(t):= [z,-t]$ and $z \in Z_{\om}$ is arbitrary. Note that for different $z\in Z_{\om}$ 
the rays $\ov{\ga}_z$ are equivalent in terms of the geodesic boundary, as $[z_1,-t]=[z_2,-t]$ for $-t\leq -\ln\rho (z_1,z_2)$. 
It is clear that the map $\iota$ is injective.

To show surjectivity let $v'\in \di X \sm\{\iota(\om)\}$ be given.
Let $\ga:[0,\infty) \to X$ be the unit speed ray with 
$\ga(0)=[u,0]$, where $u\in Z_{\om}$ is some basepoint, and $\ga(\infty) = v'$.
We know that two points can be joined by the unique natural geodesic  
as described above. Since $v'\neq \iota(\omega)$  we see that there is $s_0\in [0,\infty)$, such that 
$\ga(s_0)=[z_{s_0},0]$ and 
for $s>s_0$ it is $\gamma(s)=[z_s,s-s_0]$ for some points $z_s \in Z_{\om}$. 
Consider now numbers $s'>s>s_0$. We have $d(\ga(s'),\ga(s))=s'-s$ as well as
$d(\ga(s'),\ga(s))= s'+s-2s_0 - \min\{s-s_0,-\ln\rho(z_s,z_{s'})\}$.
Hence $-\ln\rho(z_s,z_{s'}) \geq s-s_0$. 
By completeness of $(Z,\rho)$ this implies that the points $z_s$ converge in $Z_{\om}$ for $s\to \infty$.
Let $z_s\to v \in Z_{\om}$ for $s\to \infty$.
Then $[z_s,t]=[v,t]$ for $t$ sufficiently large. Thus $\ga(s)=[v,s-s_0]$ for $s$ large which implies that
$\iota(v)=v'$.

Finally we show that
the M\"obius structures coincide.
Choose a basepoint $u \in Z_{\om}$.
We actually
show that $\iota: (Z,\rho) \to (\di X,\rho_{\omega',[u,0]})$ is an isometry,
where $\rho_{\omega',[u,0]} = e^{-(.|.)_{\om',[u,0]}}$ is the metric in the canonical M\"obius structure of $\pari X$
with remote point $\om'=\iota(\om)$.\\
For $a,b \in Z_{\om}$ we compute, where we use $a'=\iota(a)$, $b'=\iota(b)$,
\begin{align*}
(a'|b')_{\om',[u,0]}&=(a'|b')_{[u,0]} - (a'|\om')_{[u,0]} -(b'|\om')_{[u,0]}\\
&=\lim_{t\to  \infty} (\  ([a,t]|[b,t])_{[u,0]} - ([a,t]|[u,-t])_{[u,0]} -([b,t]|[u,-t])_{[u,0]} \ )\\
&=\lim_{t\to\infty}\frac{1}{2} (2t - d([a,t],[b,t])) \\
& =\lim\limits_{t\to\infty} \frac{1}{2}(2t-(2t-2\min\lbrace t,-\ln\rho(a,b)\rbrace))\\
&=-\ln\rho(a,b)
\end{align*}
which implies
$\rho_{\omega',[u,0]}(a',b')=\rho(a,b)$.

\end{proof}

%%%%%%%%%%%%%%%%%%%%%%%%%%%

\begin{lem}\label{distance-in-terms-of-crossratio}
Let $(X,d)$ be a tree and let $a,b,c_1,c_2\in\pari X$ points,
such that $a,b,c_1$ and $a,b,c_2$ are distinct. Let $u_i=Trip(a,b,c_i)$, then 
\begin{align*}
|u_1 u_2|=\pm \ln([a,c_1,c_2,b]),
\end{align*}
while we have a positive sign if $u_2\in[u_1,b)$ and a negative if $u_1\in[u_2,a)$.
\end{lem}

\begin{proof}
Let $o\in (a,b)$ s.t. $u_1,u_2\in [o,b)$ and $\rho_o$ denotes the Bourdon metric with some basepoint $o\in X$. 
Let $\mu_1=|ou_1|$ and $\mu_2=|ou_2|$, then it is $\rho_o(a,c_1)=1$, $\rho_o(a,c_2)=1$, $\rho_o(b,c_1)=e^{-\mu_1}$, $\rho_o(b,c_2)=e^{-\mu_2}$ and hence
\begin{align*}
&[a,c_1,c_2,b]=\frac{\rho_o(a,c_2)\rho_o(c_1,b)}{\rho_o(b,c_2)\rho_o(c_1,a)}=e^{-\mu_1+\mu_2}.\\
\Longrightarrow \quad & \ln([a,c_1,c_2,b])=\mu_2-\mu_1.
\end{align*}
If $\mu_1\leq \mu_2$, $\ln([a,c_1,c_2,b])=|u_1u_2|$ and if $\mu_1\geq \mu_2$, $\ln([a,c_1,c_2,b])=-|u_1u_2|$, which proves the claim.
\end{proof}

\begin{thm}
Let $X$ be a nonelementary geodesically complete tree and 
let $X^{'}$ be an arbitrary tree with $|\partial_{\infty}X'|\geq 3$. Then for
any Moebius embedding $f:\pari X\to \pari X^{'}$ w.r.t. the natural Moebius structures, there exists a unique isometric embedding $F:X\to X^{'}$.
s.t. the extension of $F$ to the boundary coincides with $f$.\\
If $f$ is in addition surjective and $X^{'}$ geodesically complete, then $F$ is an isometry.   
\end{thm}

\begin{proof}
First we define a map $F:X\to X'$. For a given $x\in X$ there are, since $X$ is nonelementary and geodesically complete, 
points $a,b\in\pari X$, s.t. $x\in (a,b)$. Let $c\in \pari X\backslash \lbrace a,b\rbrace$ and 
$\gamma_{a,b,c}$ the unit speed geodesic with $\gamma(-\infty)=a$, $\gamma(\infty)=b$ and $\gamma(0)=Trip(a,b,c)$. Then $x=\gamma(t_x)$ 
for $t_x=\pm|Trip(a,b,c)x|$, while the sign is positive if $x\in [Trip(a,b,c),b)$ and negative if $x\in (a, Trip(a,b,c)]$. 
Let $a'=f(a)$, $b'=b$ and $c'=c$, we define $F_{a,b,c}(x):=\gamma'_{a',b',c'}(t_x)$ for $\gamma'_{a',b',c'}$ 
the unit speed geodesic with $\gamma'(-\infty)=a'$, $\gamma'(\infty)=b'$ and $\gamma'(0)=Trip(a',b',c')$. 
We have to show that $F_{a,b,c}$ is independent of different choices of $a,b$ and $c$. 
Let $a_1,b_1,c_1,a_2,b_2,c_2\in\pari X$, s.t. $x\in (a_1,b_1)$, $x\in (a_2,b_2)$ and 
the three points $a_i,b_i,c_i$ are mutually distinct for $i=1,2$. In
three steps we will show the following \vskip.15mm
\textbf{Claim:} $F_{a_1,b_1,c_1}(x)=F_{a_2,b_2,c_2}(x)$.\vskip.15cm
\textbf{Step 1:} $F_{a,b,c}(x)=F_{b,a,c}(x)$\vskip.15cm
By construction it follows $\gamma_{b,a,c}(-t)=\gamma_{a,b,c}(t)$, the absolute value of $t_x$ stays the same, 
but the sign of $t_x$ changes by interchanging $a$ and $b$ in the above construction. In particular this implies $F_{a,b,c}(x)=F_{b,a,c}(x)$.\vskip.15mm
\textbf{Step 2:} If $x\in (a,b) \cap (a,b')$ then $F_{(a,b,b')}(x) = F_{(a,b',b)}(x)$  and 
if $x\in (a,b) \cap (a',b)$ then $F_{(a,b,a')}(x) = F_{(a',b,a)}(x)$\vskip.15cm
Note that $\ga_{a,b,b'}$ and $\ga_{a,b',b}$ coincide on $(-\infty,0]$ and
$\ga_{a,b,a'}$ coincides with $\ga_{a',b,a}$ on $[0,\infty)$. \vskip.15cm
\textbf{Step 3:} $F_{a,b,c_1}(x)=F_{a,b,c_2}(x)$ in the case that $c_1,c_2 \in \pari X \setminus\{a,b\}$\vskip.15cm
Let $u_i=Trip(a,b,c_i)$ and $t_x(a,b,c_i)=\pm|Trip(a,b,c_i)x|$ be the signed distances as in the construction. By construction it's 
\begin{align*}
t_x(a,b,c_1) & =t_x(a,b,c_2)- |u_1u_2|\quad \text{if} \quad u_2 \in[u_1,b)\\
t_x(a,b,c_1) & =t_x(a,b,c_2) + |u_1u_2|\quad \text{if} \quad u_1 \in[u_2,b).
\end{align*}
Let $u'_i=Trip(a',b',c'_i)$ and let $\pm|F_{a,b,c_2}(x)u'_1|$ be the signed distance being 
positive if $F_{a,b,c_2}(x)\in [u'_1,b')$ and negative if $u'_1\in [F_{a,b,c_2}(x),b')$. Due to elementary arguments it follows
\begin{align*}
\pm|F_{a,b,c_2}(x)u'_1| & =t_x(a,b,c_2)- |u'_1u'_2|\quad \text{if} \quad u'_2 \in[u'_1,b')\\
\pm|F_{a,b,c_2}(x)u'_1| & =t_x(a,b,c_2) + |u'_1u'_2|\quad \text{if} \quad u'_1 \in[u'_2,b').
\end{align*}
We have $[a,c_1,c_2,b]=[a',c'_1,c'_2,b']$, as $f$ is a M\"obius map. Using this and 
Lemma~\ref{distance-in-terms-of-crossratio}, it follows that $\pm|F_{a,b,c_2}(x)u'_1|=t_x(a,b,c_1)$. 
By construction $t_x(a,b,c_1)=\pm|F_{a,b,c_1}(x)u'_1|$. Thus $F_{a,b,c_1}(x)$ and $F_{a,b,c_2}(x)$ lie on
the geodesic $(a',b')$ with the same signed distance to the point $u'_1$, hence it follows $F_{a,b,c_1}(x)=F_{a,b,c_2}(x)$.

Using step 1 and step 3 the claim follows in the case that $\{a_1,b_1\}=\{a_2,b_2\}$.
Thus (by eventually interchanging $a_2$ and $b_2$) we can assume that $b_1 \neq b_2$.
Consider first the case $a_1 = a_2$.
Using step 2 and step 3 we get
$$  F_{a_1,b_1,c_1}(x)  = F_{a_1,b_1,b_2}(x)  =  F_{a,b_2,b_1}(x) = F_{a_1,b_2,c_2}(x) $$
and hence the claim.
Let us now assume that  $a_1,a_2,b_1,b_2$ are all distinct. Let $u= Trip(a_1,b_1,a_2)$ and $v=Trip(a_1,b_1,b_2)$.
By eventually interchanging $a_2$ and $b_2$, we can assume that the order of the points on the line $(a_1,b_1)$ are
$a_1<u<v<b_1$. Note that then $x \in (a_i,b_j)$ for every choice of $i,j$. 
Using the various equalities we obtain:
\begin{align*}
\hspace{.5cm}& F_{a_1,b_1,c_1}(x)=F_{a_1,b_1,b_2}(x)=F_{a_1,b_2,b_1}(x)= F_{b_2,a_1,b_1}(x)\\
=\: &  F_{b_2,a_1,a_2}(x)=F_{b_2,a_2,a_1}(x)=F_{a_2,b_2,a_1}(x)=F_{a_2,b_2,c_2}(x)
\end{align*}

In particular we have a well defined map $F:X\to X^{'}$ by $F(x):=F_{a,b,c}(x)$, $a,b,c\in\pari X$ distinct and $x\in (a,b)$.\\
We show that $F$ is an isometry. Let $x,y\in X$ be arbitrary, 
due to geodesically completeness there are $a,b\in\pari X$, s.t. $x,y\in (a,b)$. As the distance $|xy|$ is determined by the 
signed distances of $x$ and $y$ to $Trip(a,b,c)$ for some $c\in\pari X\backslash\lbrace a,b\rbrace$ and $F(x)$, $F(y)$ have 
the same signed distances to $Trip(a',b',c')$ by construction, it follows that $|xy|=|F(x)F(y)|$.\\
It follows immediately from the construction that for distinct points $a,b \in \di X$, $F$ maps the line
$(a,b)$ to the line $(f(a),f(b)$. Thus
restricted to the boundary $F$ equals $f$.

We show that such an $F$ is unique. Any isometry $\hat{F}:X\to X'$ with $\hat{F}_{|\di X}=f$
maps the line $(a,b)$ for $a,b\in\di X$ to the line $(f(a),f(b))$. Let $\hat{F}$ and $F$ be such isometries, 
then they differ on $(a,b)$ by a translation. If we take $c\in\di X\setminus \lbrace a,b\rbrace$ it 
is $Trip(a,b,c)=(a,b)\cap(b,c)\cap(a,c)$, which shows that $F(Trip(a,b,c))=Trip(f(a),f(b),f(c))= \hat{F}(Trip(a,b,c))$. 
In particular $F_{|(a,b)}=\hat{F}_{|(a,b)}$. But as the geodesic $(a,b)$ was arbitrary and $X$ is geodesically complete it follows $F=\hat{F}$.\\
The last step is to show that $F$ is surjective if $f$ is surjective and $X'$ 
is geodesically complete. Let $x'\in X'$ be arbitrary, then there is by geodesically 
completeness $a',b'\in \pari X'$, s.t. $x'\in (a',b')$. Let $c'\in\pari X\backslash\lbrace a',b'\rbrace$ 
then $x'$ is determined by the signed distance to $Trip(a',b',c')$. In particular this 
implies that the point $x\in (f^{-1}(a'),f^{-1}(b'))$ with the same signed distance to $Trip(f^{-1}(a'),f^{-1}(b'),f^{-1}(c'))$ is 
mapped to $x'$ by construction. We note that $f^{-1}$ exists, as $f$ is as a M\"obius map injective and assumed to be surjective. 
\end{proof}

\begin{cor}
The tree constructed in theorem~\ref{filling-thm} is unique up to isometry.
\end{cor}

\begin{proof}
Let $X,X^{'}$ be two such trees. As the boundaries with the natural M\"obius structures both equal $(Z,\cM)$ (up to M\"obius isomorphism), 
there is a M\"obius 
isomorphism $f:\pari X\to \pari X^{'}$. If we apply the theorem above the result follows.
\end{proof}

\section{Antipodal diameter 1 metrics} \label{sec:ad1}

In this section we recall some facts from \cite{Bi} and give a slight generalization of one of its result.

Let $(Z,\cM)$ be an arbitrary M\"obius space. By $\cM^a_1$ we denote the subset of antipodal, 
diameter 1 metrics. 
The set $\cM^a_1$ may be empty.

Biswas has showed in \cite{Bi} that there is a natural metric on $\cM^a_1$. 
We note that he assumed $Z$ to be compact, but for the definition of this metric that is not necessary. We recall some of the 
notation and the definition of the metric.\\
Let $(Z,\cM)$ be an arbitrary M\"obius structure and $\cM^a_1$ the subset of antipodal diameter one metrics. For $\rho_1,\rho_2\in \cM^a_1$ and $\xi\in Z$ we define
\begin{align*}
\frac{d\rho_1}{d\rho_2}(\xi):=\frac{\rho_1(\eta,\eta^{'})}{\rho_1(\xi,\eta)\rho_1(\xi,\eta^{'})}\frac{\rho_2(\xi,\eta)\rho_2(\xi,\eta^{'})}{\rho_2(\eta,\eta^{'})},
\end{align*}
with $\eta,\eta^{'}\in Z\setminus\lbrace\xi\rbrace$ arbitrary but distinct. One can show that this definition is 
independent of the choice of $\eta,\eta^{'}\in Z\setminus\lbrace\xi\rbrace$. Similar as in \cite{Bi} one can show that
\begin{align*}
d_{\cM^a_1}(\rho_1,\rho_2):=\sup_{\zeta\in Z}\ln\frac{d\rho_1}{d\rho_2}(\zeta)
\end{align*}
is a metric on $\cM^a_1$. Furthermore he shows that for a CAT(-1) space $X$ and $\rho_x,\rho_y$ Bourdon metrics, we
have $\frac{d\rho_y}{d\rho_x}=\lambda_{x,y}^2$. By the definition of $\lambda$ it follows $\ln\frac{d\rho_y}{d\rho_x}(\zeta)=B_{\zeta}(x,y)$.

\begin{thm}\label{tree-isometric-to-antipodal-diam1-metrics}
Let $(Z,\cM)$ be a ultrametric M\"obius space, and let $(X,d)$ be a geodesically complete tree 
such that $(\di X,\cM_X)=(Z,\cM)$. Let $i:(X,d)\to (\cM^a_1,d_{\cM^a_1})$ be the map sending $x$ to the Bourdon metric $\rho_x$, then $i$ is an isometry.
\end{thm}

\begin{proof}
The first step is to show that $i$ is bijective.
First observe that the geodesically completeness of $(X,d)$ implies that for every $x\in X$, the metric $\rho_x$ is antipodal and of diameter 1.
This implies that the image $i$ lies in $\cM^a_1$. As $Z$ has cardinality greater than two, it easily 
follows that $\rho_x\neq\rho_y$ for $x\neq y$, which means that $i$ is injective.\\
We show that $i$ is surjective. Let $\rho\in\cM^a_1$ be given. Let $a,b\in \pari X$, s.t. $\rho(a,b)=1$ and $c\in\pari X$ arbitrary, 
but distinct from $a$ and $b$. By lemma~\ref{diam1-antipodal-metrics-are-ultrametric} $\rho$ is an ultrametic, 
therefore  $\rho(a,c)=1$ or $\rho(b,c)=1$. 
W.l.o.g. we assume $\rho(a,c)=1$. Let $x$ be in $[Trip(a,b,c),a)$, s.t. $|xTrip(a,b,c)|=-\ln(\rho(b,c))$, 
then $(\rho_x)_{|\lbrace a,b,c\rbrace}=\rho_{|\lbrace a,b,c\rbrace}$. For $d\in \pari X\backslash\lbrace a,b,c\rbrace$ and $\mu:= \rho(c,b)=\rho_x(c,b)$, we have
by the M\"obius equivalence of  $\rho$ and $\rho_x$
\begin{align*}
&(\rho(c,d) :\rho(b,d):\mu \rho(a,d))=crt_{\rho}(c,d,a,b)=crt_{\rho_x}(c,d,a,b)\\ = & (\rho_x(c,d):\rho_x(b,d):\mu \rho_x(a,d)).
\end{align*}
Thus there is a factor $\nu >0$ such that we have the vector equality
$$ \nu \ (\rho(c,d),\rho(b,d),\mu \rho(a,d)) = (\rho_x(c,d),\rho_x(b,d),\mu \rho_x(a,d)).$$
We will show that $\nu =1$.
As $\rho$ and $\rho_x$ are 
ultrametrics we have $\rho(b,d)=1$ or $\rho(a,d)=1$, as well as $\rho_x(b,d)=1$ or $\rho_x(a,d)=1$. 
If $\rho_x(a,d)=1\: \wedge\: \rho(a,d)=1$ then $\nu =1$ and 
in the same way $\rho_x(b,d)=1 \:\wedge\: \rho(b,d)=1$ implies $\nu =1$.
One easily sees that in the remaining cases
\begin{align*}
\rho_x(a,d)=1\:\wedge\:\rho(a,d)<1\:\wedge\:\rho(b,d)=1\:\wedge\:\rho_x(b,d)<1
\end{align*}
and
\begin{align*}
\rho_x(a,d)<1\:\wedge\:\rho(a,d)=1\:\wedge\:\rho(b,d)<1\:\wedge\:\rho_x(b,d)=1
\end{align*}
the vector equality from above can not be satisfied.
Thus $\nu = 1$, which implies that $(\rho_x)_{|\lbrace a,b,c,d\rbrace}=\rho_{|\lbrace a,b,c,d\rbrace}$.
In particular this means that the construction of $x$ is 
independent of the choice $c$ or $d$. Furthermore it follows that for arbitrary $c,d\in\pari X$ we have $\rho_x(c,d)=\rho(c,d)$. 
Hence $\rho$ is the 
Bourdon metric $\rho_x$.

Finally we have to show
$d(x,y)=d_{\cM^a_1}(\rho_x,\rho_y)$. For given points $x,y\in X$ we can extend the geodesic segment $[x,y]$ to a geodesic ray $[x,\xi)$ where $\xi\in \pari X$. 
Thus $d(x,y)=B_{\xi}(x,y)$. Since $B_{\zeta}(x,y)\leq d(x,y)$ for all $\zeta\in\di X$ we have
\begin{align*}
d_{\cM^a_1}(\rho_x,\rho_y)=\sup_{\zeta\in \pari X}\ln\frac{d\rho_y}{d\rho_x} (\zeta)=\sup_{\zeta\in \pari X}B_{\zeta}(x,y)=B_{\xi}(x,y)=d(x,y).
\end{align*}
\end{proof}

%%%%%%%%%%%%%%%%%%%%%%%%%%%%%%%%%%%%%%%%%%%%%%%%%%%%%%%%%%%%%%%%%%%%%%%%%%%%%%%

\bigskip
\begin{tabbing}

Jonas Beyrer \\ 
Viktor Schroeder\\

\\
Institut f\"ur Mathematik \\

Universit\"at Z\"urich \\Winterthurer Strasse 190 \\

 CH-8057 Z\"urich, Switzerland\\
 \\
 {\tt jonas.beyrer@math.uzh.ch} \\
 {\tt viktor.schroeder@math.uzh.ch}\\

\end{tabbing}


\begin{thebibliography}{ABC}


\bibitem[Bi]{Bi} K.~Biswas, On Moebius and Conformal Maps Between Boundaries of CAT(-1) Spaces, arXiv:1203.6212v3 [math.DS] (2013)

\bibitem[Bou]{Bou} M.~Bourdon,  Structure conforme au bord et flot g\'eod\'esique d'un  CAT(-1)-espace, L'Einseignement Math\'ematique 41 (1995) 63--102

\bibitem[BS]{BS} S.~Buyalo, V.~Schroeder, Elements of asymptotic geometry,
EMS Monographs in Mathematics, 2007, 209 pages.


\bibitem[FS]{FS} T.~Foertsch, V.~Schroeder, Hyperbolicity,
$\CAT(-1)$-spaces and Ptolemy inequality, Math. Ann. 350 (2011), no. 2, 339--356.

\bibitem[Hu]{Hu} B.~Hughes, Trees and ultrametric spaces: a categorical equivalence, Advances in Mathematics 189 (2004) 148 - 191



\end{thebibliography}
\end{document}